\documentclass{amsart}
\usepackage{amssymb}
\usepackage{amsmath}
\usepackage{amsfonts}

\newtheorem{theorem}{Theorem}
\theoremstyle{plain}

\newtheorem{corollary}{Corollary}

\newtheorem{definition}{Definition}
\newtheorem{example}{Example}

\newtheorem{lemma}{Lemma}

\newtheorem{proposition}{Proposition}
\newtheorem{remark}{Remark}

\numberwithin{equation}{section}
\input{tcilatex}

\begin{document}
\title{ANTI-INVARIANT RIEMANNIAN SUBMERSIONS FROM KENMOTSU MANIFOLDS ONTO
RIEMANNIAN MANIFOLDS}
\author{A. Beri}
\email{ayseyilmazsoyberi@gmail.com}
\author{I. K\"{u}peli Erken}
\address{Art and Science Faculty,Department of Mathematics, Uludag
University, 16059 Bursa, TURKEY}
\email{iremkupeli@uludag.edu.tr}
\author{C. Murathan}
\address{Uludag University, Faculty of Art and Science, Department of
Mathematics, Gorukle 16059 Bursa, TURKEY}
\email{cengiz@uludag.edu.tr}
\date{26.08.2015}

\begin{abstract}
The purpose of this paper is to study anti-invariant Riemannian submersions
from Kenmotsu manifolds onto Riemannian manifolds. Several fundamental
results in this respect are proved. The integrability of the distributions
and the geometry of foliations are investigated. We proved that there do not
exist (anti-invariant) Riemannian submersions from Kenmotsu manifolds onto
Riemannian manifolds such that characteristic vector field $\xi $ is a
vertical vector field. We gave a method to get horizontally conformal
submersion examples from warped product manifolds onto Riemannian manifolds.
Furthermore, we presented an example of anti-invariant Riemannian
submersions in the case where the characteristic vector field $\xi $ is a
horizontal vector field and an anti-invariant horizontally conformal
submersion such that $\xi $ is a vertical vector field.
\end{abstract}

\maketitle

\section{\textbf{Introduction}}

Riemannian submersions between Riemannian manifolds were studied by O'Neill 
\cite{BO1} and Gray \cite{GRAY}. Riemannian submersions have several
applications in mathematical physics. Indeed, Riemannian submersions have
their applications in the Yang-Mills theory (\cite{BL}, \cite{WATSON}),
Kaluza-Klein theory (\cite{BOL}, \cite{IV}), supergravity and superstring
theories (\cite{IV2}, \cite{MUS}), etc. Later such submersions were
considered between manifolds with differentiable structures, see \cite{FAL}.
Furthermore, we have the following submersions: semi-Riemannian submersion
and Lorentzian submersion \cite{FAL}, Riemannian submersion \cite{GRAY},
slant submersion (\cite{CHEN}, \cite{SAHIN1}), almost Hermitian submersion 
\cite{WAT}, contact-complex submersion \cite{IANUS}, quaternionic submersion 
\cite{IANUS2}, almost $h$-slant submersion and $h$-slant submersion \cite%
{PARK1}, semi-invariant submersion \cite{SAHIN2}, $h$-semi-invariant
submersion \cite{PARK2}, etc.

Comparing with the huge literature in Riemannian submersions, it seems that
there are necessary new studies in anti-invariant Riemannian submersions; an
interesting paper connecting these fields is \cite{SAHIN}. \c{S}ahin \cite%
{SAHIN} introduced anti-invariant Riemannian submersions from almost
Hermitian manifolds onto Riemannian manifolds. Later, he suggested to
investigate anti-invariant Riemannian submersions from almost contact metric
manifolds onto Riemannian manifolds \cite{SAHIN3}. The present work is
another step in this direction, more precisely from the point of view of
anti-invariant Riemannian submersions from Kenmotsu manifolds. Our work is
structured as follows: Section $2$ is focused on basic facts for Riemannian
submersions and Kenmotsu manifolds. The third section is concerned with
definition of anti-invariant Riemannian submersions from Kenmotsu manifolds
onto Riemannian manifolds. We investigate the integrability of the
distributions and the geometry of foliations. We proved that there do not
exist (anti-invariant) Riemannian submersions from Kenmotsu manifolds onto
Riemannian manifolds such that characteristic vector field $\xi $ is
vertical vector field. The last section is devoted to an example of
anti-invariant Riemannian submersions in the case where the characteristic
vector field $\xi $ is a horizontal vector field and an anti-invariant
horizontally conformal submersion such that $\xi $ is a vertical vector
field.

\section{Preliminaries}

In this section we recall several notions and results which will be needed
throughout the paper.

Let $M$ be an $(2m+1)$-dimensional connected differentiable manifold \cite%
{BLAIR} endowed with an almost contact metric structure $(\phi ,\xi ,\eta
,g) $ consisting of a $(1,1)$-tensor field $\phi ,$a vector field $\xi ,$ a $%
1$-form $\eta $ and a compatible Riemannian metric $g$ satisfying%
\begin{eqnarray}
\phi ^{2} &=&-I+\eta \otimes \xi ,\text{ \ }\phi \xi =0,\text{ }\eta \circ
\phi =0,\text{ \ \ }\eta (\xi )=1,  \label{fi} \\
g(\phi X,\phi Y) &=&g(X,Y)-\eta (X)\eta (Y),  \label{metric} \\
g(\phi X,Y)+g(X,\phi Y) &=&0,\text{\ }\eta (X)=g(X,\xi ),  \label{gfi}
\end{eqnarray}%
for all vector fields $X,Y\in \chi (M).$

An almost contact metric manifold $M$ is said to be a Kenmotsu manifold \cite%
{kenmotsu} if it satisfies%
\begin{equation}
(\nabla _{X}\phi )Y=g(\phi X,Y)\xi -\eta (Y)\phi X,  \label{nambla}
\end{equation}%
where $\nabla $ is Levi-Civita connection of the Riemannian metric $g$. From
the above equation it follows that%
\begin{eqnarray}
\nabla _{X}\xi &=&X-\eta (X)\xi ,  \label{xzeta} \\
(\nabla _{X}\eta )Y &=&g(X,Y)-\eta (X)\eta (Y).  \label{gereksiz}
\end{eqnarray}

A Kenmotsu manifold is normal (that is, the Nijenhuis tensor of $\phi $
equals $-2d\eta \otimes \xi $) but not Sasakian. Moreover, it is also not
compact since from equation (\ref{xzeta}) we get div$\xi =2m$. Finally, the
fundamental $2$-form $\Phi $ is defined by $\Phi (X,Y)=g(X,\phi Y).$ In \cite%
{kenmotsu}, Kenmotsu showed:

$(a)$ that locally a Kenmotsu manifold is a warped product $I\times _{f}N$
of an interval $I$ and a Kaehler manifold $N$ with warping function $%
f(t)=se^{t},$ where $s$ is a nonzero constant.

$(b)$ that a Kenmotsu manifold of constant $\phi $-sectional curvature is a
space of constant curvature $-1$ and so it is locally hyperbolic space.

Now we will give a well known example which is Kenmotsu manifold on $%
\mathbb{R}
^{5}$ by using $(a).$

\begin{example}
\label{KEN}We consider $M=\{(x_{1},x_{2},y_{1},y_{2},z)\in 
\mathbb{R}
^{5}:$ $z\neq 0\}.$ Let $\eta $ be a $1$-form defined by%
\begin{equation*}
\eta =dz.
\end{equation*}%
The characteristic vector field $\xi $ is given by $\frac{\partial }{%
\partial z}$ and its Riemannian metric $g$ in and tensor field $\phi $ are
given by%
\begin{equation*}
g=e^{2z}\sum\limits_{i=1}^{2}((dx_{i})^{2}+(dy_{i})^{2})+(dz)^{2},\text{ \ }%
\phi =\left( 
\begin{array}{ccccc}
0 & 0 & -1 & 0 & 0 \\ 
0 & 0 & 0 & -1 & 0 \\ 
1 & 0 & 0 & 0 & 0 \\ 
0 & 1 & 0 & 0 & 0 \\ 
0 & 0 & 0 & 0 & 0%
\end{array}%
\right)
\end{equation*}%
This gives a Kenmotsu structure on $M$. The vector fields $E_{1}=e^{-z}\frac{%
\partial }{\partial y_{1}},$ $E_{2}=e^{-z}\frac{\partial }{\partial y_{2}},$ 
$E_{3}=e^{-z}\frac{\partial }{\partial x_{1}}$, $E_{4}=e^{-z}\frac{\partial 
}{\partial x_{2}}$ and $E_{5}=$ $\xi $ form a $\phi $-basis for the Kenmotsu
structure. On the other hand, it can be shown that $M(\phi ,\xi ,\eta ,g)$
is a Kenmotsu manifold.
\end{example}

Let $(M,g_{M})$ be an $m$-dimensional Riemannian manifold and let $(N,g_{N})$
be an $n$-dimensional Riemannian manifold. A Riemannian submersion is a
smooth map $F:M\rightarrow N$ which is onto and satisfying the following
axioms:

$S1$. $F$ has maximal rank.

$S2$. The differential $F_{\ast }$ preserves the lenghts of horizontal
vectors.

The fundamental tensors of a submersion were defined by O'Neill (\cite{BO1},%
\cite{BO2}). They are $(1,2)$-tensors on $M$, given by the following
formulas:%
\begin{eqnarray}
\mathcal{T}(E,F) &=&\mathcal{T}_{E}F=\mathcal{H}\nabla _{\mathcal{V}E}%
\mathcal{V}F+\mathcal{V}\nabla _{\mathcal{V}E}\mathcal{H}F,  \label{AT1} \\
\mathcal{A}(E,F) &=&\mathcal{A}_{E}F=\mathcal{V}\nabla _{\mathcal{H}E}%
\mathcal{H}F+\mathcal{H}\nabla _{\mathcal{H}E}\mathcal{V}F,  \label{AT2}
\end{eqnarray}%
for any vector fields $E$ and $F$ on $M.$ Here $\nabla $ denotes the
Levi-Civita connection of $g_{M}$. These tensors are called integrability
tensors for the Riemannian submersions. Note that we denote the projection
morphism on the distributions ker$F_{\ast }$ and $($ker$F_{\ast })^{\perp }$
by $\mathcal{V}$ and $\mathcal{H},$ respectively.

If the second condition $S2.$ can be changed as $F_{\ast }$ restricted to
horizontal distribution of $F$ is a conformal mapping, we get \textit{%
horizontally} \textit{conformal submersion} definition \cite{ORN}. In this
case the second condition can be written in a following way:%
\begin{equation}
g_{M}(X,Y)=e^{2\lambda (p)}g_{N}(F_{\ast }X,F_{\ast }Y),\forall p\in
M,\forall X,Y\in \Gamma ((\ker F_{\ast })^{\bot })\text{, }\exists \lambda
\in C^{\infty }(M).  \label{CONFORMAL}
\end{equation}%
The warped product $M=M_{1}\times _{f}M_{2}$ of two Riemannian manifolds ($%
M_{1}$,$g_{1}$) and ($M_{2}$,$g_{2}$), is the Cartesian product manifold $%
M_{1}\times M_{2},$ endowed with the warped product metric $g=g_{1}+fg_{2}$,
where $f$ is a positive function on $M_{1}$. More precisely, the Riemannian
metric $g$ on $M_{1}\times _{f}M_{2}$ is defined for pairs of vector fields $%
X,Y$ on $M_{1}\times M_{2}$ by%
\begin{equation*}
g(X,Y)=g_{1}(\pi _{1\ast }(X),\pi _{1\ast }(Y))+f^{2}(\pi _{1}(.))g_{2}(\pi
_{2\ast }(X),\pi _{2\ast }(Y)),
\end{equation*}%
where $\pi _{1}:M_{1}\times M_{2}\rightarrow M_{1};(p,q)\rightarrow p$ and $%
\pi _{2}:M_{1}\times M_{2}\rightarrow M_{2};(p,q)\rightarrow q$ are the
canonical projections. We recall that this projections are submersions. If $%
f $ is not a constant function of value $1$, one can prove that second
projection is a conformal submersion whose vertical and horizontal spaces at
any point $(p,q)$ are respectively identified with $T_{p}M_{1},T_{q}M_{2}.$

Let $\mathcal{L}(M_{1})$ and $\mathcal{L}(M_{2})$ be the set of lifts of
vector fields on $M_{1}\ $and $M_{2}$ to $M_{1}\times _{f}M_{2}$
respectively. We use the same notation for a vector field and for its lift.
We denote the Levi-Civita connection of the warped product metric tensor of $%
g$ by $\nabla $.

\begin{proposition}
\cite{BO2}\label{ONEILL}$M=M_{1}\times _{f}M_{2}$ be a warped Riemannian
product manifold with the warping function $f$ on $M_{1}$. If $%
X_{1},Y_{1}\in \mathcal{L}(M_{1})$ and $X_{2},Y_{2}\in \mathcal{L}(M_{2})$,
then

(i)$\nabla _{X_{1}}Y_{1}$ is the lift of $\ \nabla _{X_{1}}^{1}Y_{1},$

(ii)$\nabla _{X_{1}}X_{2}$ $=$ $\nabla _{X_{2}}X_{1}=(X_{1}f/f)X_{2},$

(iii) nor $\nabla _{X_{2}}Y_{2}$ $=-(g(X_{2},Y_{2})/f)gradf$,

(iv) tan $\nabla _{X_{2}}Y_{2}\in \mathcal{L}(M_{2})$ is the lift of $\
\nabla _{X_{2}}^{2}Y_{2},$

where $\nabla ^{1}$ and $\nabla ^{2}$ are Riemannian connections on $M_{1}$
and $M_{2}$, respectively.
\end{proposition}

Now we will introduce the following proposition (\cite{BRICLC},pp.86) for
the Subsection 3.2.

\begin{proposition}
\label{Clark} If $\phi $ is a submersion of $N$ onto $N_{1}$ and if $\psi $:$%
N_{1}\rightarrow N_{2}$ is a differentiable function, then the rank of $\psi
\circ \phi $ at $p$ is equal to the rank of $\psi $ at $\phi (p).$
\end{proposition}

The following lemmas are well known from (\cite{BO1},\cite{BO2}):

\begin{lemma}
\label{t}For any $U,W$ vertical and $X,Y$ horizontal vector fields, the
tensor fields $\mathcal{T}$ and $\mathcal{A}$ satisfy%
\begin{eqnarray}
i)\mathcal{T}_{U}W &=&\mathcal{T}_{W}U,  \label{TUW} \\
ii)\mathcal{A}_{X}Y &=&-\mathcal{A}_{Y}X=\frac{1}{2}\mathcal{V}\left[ X,Y%
\right] .  \label{TUW2}
\end{eqnarray}
\end{lemma}

It is easy to see that $\mathcal{T}$ $\ $is vertical, $\mathcal{T}_{E}=%
\mathcal{T}_{\mathcal{V}E},$ and $\mathcal{A}$ is horizontal, $\mathcal{A=A}%
_{\mathcal{H}E}$.

For each $q\in N,$ $F^{-1}(q)$ is an $(m-n)$-dimensional submanifold of $M$.
The submanifolds $F^{-1}(q)$ are called fibers. A vector field on $M$ is
called vertical if it is always tangent to fibers. A vector field on $M$ is
called horizontal if it is always orthogonal to fibers. A vector field $X$
on $M$ is called basic if $X$ is horizontal and $F$-related to a vector
field $X_{\ast }$ on $N,$ i. e., $F_{\ast }X_{p}=X_{\ast F(p)}$ for all $%
p\in M.$

\begin{lemma}
\label{basic}Let $F:(M,g_{M})\rightarrow (N,g_{N})$ be a Riemannian
submersion. If $\ X,$ $Y$ are basic vector fields on $M$, then
\end{lemma}

$i)$ $g_{M}(X,Y)=g_{N}(X_{\ast },Y_{\ast })\circ F,$

$ii)$ $\mathcal{H}[X,Y]$ is basic and $F$-related to $[X_{\ast },Y_{\ast }]$,

$iii)$ $\mathcal{H}(\nabla _{X}Y)$ is a basic vector field corresponding to $%
\nabla _{X_{\ast }}^{^{\ast }}Y_{\ast }$ where $\nabla ^{\ast }$ is the
connection on $N.$

$iv)$ for any vertical vector field $V$, $[X,V]$ is vertical.

Moreover, if $X$ is basic and $U$ is vertical, then $\mathcal{H}(\nabla
_{U}X)=\mathcal{H}(\nabla _{X}U)=\mathcal{A}_{X}U.$ On the other hand, from (%
\ref{AT1}) and (\ref{AT2}) we have%
\begin{eqnarray}
\nabla _{V}W &=&\mathcal{T}_{V}W+\hat{\nabla}_{V}W,  \label{1} \\
\nabla _{V}X &=&\mathcal{H\nabla }_{V}X+\mathcal{T}_{V}X,  \label{2} \\
\nabla _{X}V &=&\mathcal{A}_{X}V+\mathcal{V}\nabla _{X}V,  \label{3} \\
\nabla _{X}Y &=&\mathcal{H\nabla }_{X}Y+\mathcal{A}_{X}Y,  \label{4}
\end{eqnarray}%
for $X,Y\in \Gamma ((\ker F_{\ast })^{\bot })$ and $V,W\in \Gamma (\ker
F_{\ast }),$ where $\hat{\nabla}_{V}W=\mathcal{V}\nabla _{V}W.$

Notice that $\mathcal{T}$ acts on the fibres as the second fundamental form
of the submersion and restricted to vertical vector fields and it can be
easily seen that $\mathcal{T}=0$ is equivalent to the condition that the
fibres are totally geodesic. A Riemannian submersion is called a Riemannian
submersion with totally geodesic fibers if $\mathcal{T}$ $\ $vanishes
identically. Let $U_{1},...,U_{m-n}$ be an orthonormal frame of $\Gamma
(\ker F_{\ast }).$ Then the horizontal vector field $H$ $=\frac{1}{m-n}%
\sum\limits_{j=1}^{m-n}\mathcal{T}_{U_{j}}U_{j}$ is called the mean
curvature vector field of the fiber. If \ $H$ $=0,$ then the Riemannian
submersion is said to be minimal. A Riemannian submersion is called a
Riemannian submersion with totally umbilical fibers if 
\begin{equation}
\mathcal{T}_{U}W=g_{M}(U,W)H,  \label{4a}
\end{equation}%
for $U,W\in $ $\Gamma (\ker F_{\ast })$. For any $E\in \Gamma (TM),\mathcal{T%
}_{E\text{ }}$and $\mathcal{A}_{E}$ are skew-symmetric operators on $(\Gamma
(TM),g_{M})$ reversing the horizontal and the vertical distributions. By
Lemma $1,$ horizontal distribution $\mathcal{H}$ is integrable if and only
if \ $\mathcal{A=}0$. For any $D,E,G\in \Gamma (TM),$ one has%
\begin{equation}
g(\mathcal{T}_{D}E,G)+g(\mathcal{T}_{D}G,E)=0  \label{4b}
\end{equation}%
and%
\begin{equation}
g(\mathcal{A}_{D}E,G)+g(\mathcal{A}_{D}G,E)=0.  \label{4c}
\end{equation}

Finally, we recall the notion of harmonic maps between Riemannian manifolds.
Let $(M,g_{M})$ and $(N,g_{N})$ be Riemannian manifolds and supposed that $%
\varphi :M\rightarrow N$ is a smooth map between them. Then the differential 
$\varphi _{\ast }$ of $\varphi $ can be viewed as a section of the bundle $\
Hom(TM,\varphi ^{-1}TN)\rightarrow M,$ where $\varphi ^{-1}TN$ is the
pullback bundle which has fibres $(\varphi ^{-1}TN)_{p}=T_{\varphi (p)}N,$ $%
p\in M.\ Hom(TM,\varphi ^{-1}TN)$ has a connection $\nabla $ induced from
the Levi-Civita connection $\nabla ^{M}$ and the pullback connection. Then
the second fundamental form of $\varphi $ is given by 
\begin{equation}
(\nabla \varphi _{\ast })(X,Y)=\nabla _{X}^{\varphi }\varphi _{\ast
}(Y)-\varphi _{\ast }(\nabla _{X}^{M}Y),  \label{5}
\end{equation}%
for $X,Y\in \Gamma (TM),$ where $\nabla ^{\varphi }$ is the pullback
connection. It is known that the second fundamental form is symmetric. If $%
\varphi $ is a Riemannian submersion, it can be easily proved that 
\begin{equation}
(\nabla \varphi _{\ast })(X,Y)=0,  \label{5a}
\end{equation}%
for $X,Y\in \Gamma ((\ker F_{\ast })^{\bot })$. A smooth map $\varphi
:(M,g_{M})\rightarrow (N,g_{N})$ is said to be harmonic if $trace(\nabla
\varphi _{\ast })=0.$ On the other hand, the tension field of $\varphi $ is
the section $\tau (\varphi )$ of $\Gamma (\varphi ^{-1}TN)$ defined by%
\begin{equation}
\tau (\varphi )=div\varphi _{\ast }=\sum_{i=1}^{m}(\nabla \varphi _{\ast
})(e_{i},e_{i}),  \label{6}
\end{equation}%
where $\left\{ e_{1},...,e_{m}\right\} $ is the orthonormal frame on $M$.
Then it follows that $\varphi $ is harmonic if and only if $\tau (\varphi
)=0 $, (for details, see \cite{B}).

Let $g$ be a Riemannian metric tensor on the manifold $M=M_{1}\times M_{2}$
and assume that the canonical foliations \ $D_{M_{1}}$ and $D_{M_{2}}$
intersect perpendicularly everywhere. Then $g$ is the metric tensor of a
usual product of Riemannian manifolds if and only if $D_{M_{1}}$ and $%
D_{M_{2}}$ are totally geodesic foliations \cite{PON}.

\section{\textbf{Anti-invariant Riemannian submersions}}

In this section, we are going to define anti-invariant Riemannian
submersions from Kenmotsu manifolds and investigate the geometry of such
submersions.

\begin{definition}
\label{anti}Let $M(\phi ,\xi ,\eta ,g_{M})$ be a Kenmotsu manifold and $%
(N,g_{N})$ a Riemannian manifold. A Riemannian submersion $F:M(\phi ,\xi
,\eta ,g_{M})\rightarrow $ $(N,g_{N})$ is called an anti-invariant
Riemannian submersion if $\ker F_{\ast }$ is anti-invariant with respect to $%
\phi $, i.e. $\phi (\ker F_{\ast })\subseteq (\ker F_{\ast })^{\bot }.$
\end{definition}

Let $F:M(\phi ,\xi ,\eta ,g_{M})\rightarrow $ $(N,g_{N})$ be an
anti-invariant Riemannian submersion from a Kenmotsu manifold $M(\phi ,\xi
,\eta ,g_{M})$ to a Riemannian manifold $(N,g_{N}).$ First of all, from
Definition \ref{anti}, we have $\phi (\ker F_{\ast })^{\bot }\cap (\ker
F_{\ast })\neq \left\{ 0\right\} .$ We denote the complementary orthogonal
distribution to $\phi (\ker F_{\ast })$ in $(\ker F_{\ast })^{\bot }$ by $%
\mu .$ Then we have%
\begin{equation}
(\ker F_{\ast })^{\bot }=\phi \ker F_{\ast }\oplus \mu .  \label{A1}
\end{equation}

\subsection{\textbf{Anti-invariant Riemannian submersions admitting
horizontal structure vector field }}

In this subsection, we will study anti-invariant Riemannian submersions from
a Kenmotsu manifold onto a Riemannian manifold such that the characteristic
vector field $\xi $ is a horizontal vector field. Using (\ref{A1}), we have $%
\mu =\phi \mu \oplus \{\xi \}.$ For any horizontal vector field $X$ we put 
\begin{equation}
\phi X=BX+CX,  \label{IREM}
\end{equation}%
where $BX\in \Gamma (\ker F_{\ast })$ and $CX\in \Gamma (\mu ).$

Now we suppose that $V$ is vertical and $X$ is horizontal vector field.
Using above relation and (\ref{metric}), we obtain%
\begin{equation}
g_{M}(CX,\phi V)=0.  \label{IKE3}
\end{equation}%
By virtue of (\ref{metric}) and (\ref{IREM}), we get 
\begin{eqnarray}
g_{M}(CX,\phi U) &=&g_{M}(\phi X-BX,\phi U)  \label{6.5} \\
&=&g_{M}(X,U)-\eta (X)\eta (U)-g_{M}(BX,\phi U).  \notag
\end{eqnarray}%
Since $\phi U\in \Gamma ((\ker F_{\ast })^{\bot })$ and $\xi \in \Gamma
(\ker F_{\ast })^{\bot },$ (\ref{6.5}) implies (\ref{IKE3}). From this last
relation we have $g_{N}(F_{\ast }\phi V,F_{\ast }CX)=0$ which implies that 
\begin{equation}
TN=F_{\ast }(\phi (\ker F_{\ast }))\oplus F_{\ast }(\mu ).  \label{Ba}
\end{equation}%
The proof of the following result is the same as Theorem $10$ of \cite{cm},
therefore we omit its proof.

\begin{theorem}
\label{dimhor}Let $M(\phi ,\xi ,\eta ,g_{M})$ be a Kenmotsu manifold \ of
dimension $2m+1$ and $(N,g_{N})$ a Riemannian manifold of dimension $n.$ Let 
$F:M(\phi ,\xi ,\eta ,g_{M})\rightarrow $ $(N,g_{N})$ be an anti-invariant
Riemannian submersion such that $(\ker F_{\ast })^{\bot }=\phi \ker F_{\ast
}\oplus \{\xi \}.$Then $m+1=n$.
\end{theorem}

\begin{remark}
\label{remhor}We note that Example \ref{ex2} satisfies Theorem \ref{dimhor}.
\end{remark}

\begin{lemma}
\label{horAT}Let $F$ be an anti-invariant Riemannian submersion from a
Kenmotsu manifold $M(\phi ,\xi ,\eta ,g_{M})$ to a Riemannian manifold $%
(N,g_{N})$. Then we have%
\begin{equation}
\mathcal{A}_{X}\xi =0,  \label{IKE1}
\end{equation}%
\begin{equation}
\mathcal{T}_{U}\xi =U,  \label{IKE2}
\end{equation}%
\begin{equation}
g_{M}(\nabla _{Y}CX,\phi U)=-g_{M}(CX,\phi \mathcal{A}_{Y}U),  \label{IKE4}
\end{equation}%
for $X,Y\in \Gamma ((\ker F_{\ast })^{\bot })$ and $U\in \Gamma (\ker
F_{\ast }).$
\end{lemma}

\begin{proof}
Using (\ref{4}) and (\ref{xzeta}), we have (\ref{IKE1}). Using (\ref{2}) and
(\ref{xzeta}), we obtain (\ref{IKE2}). Now using (\ref{IKE3}), we get%
\begin{equation*}
g_{M}(\nabla _{Y}CX,\phi U)=-g_{M}(CX,\nabla _{Y}\phi U),
\end{equation*}%
for $X,Y\in \Gamma ((\ker F_{\ast })^{\bot })$ and $U\in \Gamma (\ker
F_{\ast })$. Then (\ref{3}) and (\ref{nambla}) imply that 
\begin{equation*}
g_{M}(\nabla _{Y}CX,\phi U)=-g_{M}(CX,\phi \mathcal{A}_{Y}U)-g_{M}(CX,\phi (%
\mathcal{V}\nabla _{Y}U)).
\end{equation*}%
Since $\phi (\mathcal{V}\nabla _{Y}U)\in \Gamma ((\ker F_{\ast })^{\bot }),$
we obtain (\ref{IKE4}).
\end{proof}

We now study the integrability of the distribution $(\ker F_{\ast })^{\bot }$
and then we investigate the geometry of leaves of $\ker F_{\ast }$ and $%
(\ker F_{\ast })^{\bot }$.

\begin{theorem}
\label{teohor}Let $F$ be an anti-invariant Riemannian submersion from a
Kenmotsu manifold $M(\phi ,\xi ,\eta ,g_{M})$ to a Riemannian manifold $%
(N,g_{N})$. Then the following assertions are equivalent to each other:
\end{theorem}

$\ i)$ $(\ker F_{\ast })^{\bot }$ \textit{is integrable,}

$\ ii)~g_{N}((\nabla F_{\ast })(Y,BX),F_{\ast }\phi V)=g_{N}((\nabla F_{\ast
})(X,BX),F_{\ast }\phi V)$

\ \ \ \ \ \ \ \ \ \ \ \ \ \ \ \ \ \ \ \ \ \ \ \ \ \ \ \ \ \ \ \ \ \ \ \ \ \
\ \ \ \ \ \ $+g_{M}(CY,\phi \mathcal{A}_{X}V)-g_{M}(CX,\phi \mathcal{A}%
_{Y}V),$

$iii)$ $g_{M}(\mathcal{A}_{X}BY-\mathcal{A}_{Y}BX,\phi V)=g_{M}(CY,\phi 
\mathcal{A}_{X}V)-g_{M}(CX,\phi \mathcal{A}_{Y}V)$

\textit{for }$X,Y\in \Gamma ((\ker F_{\ast })^{\bot })$\textit{\ and }$V\in
\Gamma (\ker F_{\ast }).$

\begin{proof}
From (\ref{metric}) and (\ref{nambla}), one easily obtains%
\begin{eqnarray*}
g_{M}(\left[ X,Y\right] ,V) &=&g_{M}(\nabla _{X}Y,V)-g_{M}(\nabla _{Y}X,V) \\
&=&g_{M}(\nabla _{X}\phi Y,\phi V)-g_{M}(\nabla _{Y}\phi X,\phi V).
\end{eqnarray*}
for $X,Y\in \Gamma ((\ker F_{\ast })^{\bot })$ and $V\in \Gamma (\ker
F_{\ast }).$Then from (\ref{IREM}), we have%
\begin{eqnarray*}
g_{M}(\left[ X,Y\right] ,V) &=&g_{M}(\nabla _{X}BY,\phi V)+g_{M}(\nabla
_{X}CY,\phi V)-g_{M}(\nabla _{Y}BX,\phi V) \\
&&-g_{M}(\nabla _{Y}CX,\phi V).
\end{eqnarray*}%
Taking into account that $F$ is a Riemannian submersion and using (\ref{AT2}%
), (\ref{3}) and (\ref{IKE4}), we obtain%
\begin{eqnarray*}
g_{M}(\left[ X,Y\right] ,V) &=&g_{N}(F_{\ast }\nabla _{X}BY,F_{\ast }\phi
V)-g_{M}(CY,\phi \mathcal{A}_{X}V) \\
&&-g_{N}(F_{\ast }\nabla _{Y}BX,F_{\ast }\phi V)+g_{M}(CX,\phi \mathcal{A}%
_{Y}V).
\end{eqnarray*}%
Thus, from (\ref{5}) we have 
\begin{eqnarray*}
g_{M}(\left[ X,Y\right] ,V) &=&g_{N}(-(\nabla F_{\ast })(X,BY)+(\nabla
F_{\ast })(Y,BX),F_{\ast }\phi V) \\
&&+g_{M}(CX,\phi \mathcal{A}_{Y}V)-g_{M}(CY,\phi \mathcal{A}_{X}V)
\end{eqnarray*}%
which proves $(i)\Leftrightarrow (ii)$. On the other hand using (\ref{5}),
we get%
\begin{equation*}
(\nabla F_{\ast })(Y,BX)-(\nabla F_{\ast })(X,BY)=-F_{\ast }(\nabla
_{Y}BX-\nabla _{X}BY).
\end{equation*}%
Then (\ref{3}) implies that 
\begin{equation*}
(\nabla F_{\ast })(Y,BX)-(\nabla F_{\ast })(X,BY)=-F_{\ast }(\mathcal{A}%
_{Y}BX-\mathcal{A}_{X}BY).
\end{equation*}%
From (\ref{AT2}) it follows that $\mathcal{A}_{Y}BX-\mathcal{A}_{X}BY\in
\Gamma ((\ker F_{\ast })^{\bot }),$ this shows that $(ii)\Leftrightarrow
(iii).$
\end{proof}

\begin{remark}
\label{remczero}We assume that $(\ker F_{\ast })^{\bot }=\phi \ker F_{\ast
}\oplus \{\xi \}.$ Using (\ref{IREM}) one can prove that $CX=0$ \textit{for }%
$X\in \Gamma ((\ker F_{\ast })^{\bot })$.
\end{remark}

Hence we can give the following corollary.

\begin{corollary}
\label{inthorcor}Let $M(\phi ,\xi ,\eta ,g_{M})$ be a Kenmotsu manifold \ of
dimension $2m+1$ and $(N,g_{N})$ a Riemannian manifold of dimension $n.$ Let 
$F:M(\phi ,\xi ,\eta ,g_{M})\rightarrow $ $(N,g_{N})$ be an anti-invariant
Riemannian submersion such that $(\ker F_{\ast })^{\bot }=\phi \ker F_{\ast
}\oplus \{\xi \}.$ Then the following assertions are equivalent to each
other:

$i)$ $(\ker F_{\ast })^{\bot }$ \textit{is integrable,}

$ii)~(\nabla F_{\ast })(X,\phi Y)=(\nabla F_{\ast })(\phi X,Y),$ for $X\in
\Gamma ((\ker F_{\ast })^{\bot })$ and $X,Y\in \Gamma ((\ker F_{\ast
})^{\bot }),$

$iii)$ $\mathcal{A}_{X}\phi Y=\mathcal{A}_{Y}\phi X.$
\end{corollary}

\begin{theorem}
\label{totgeohor}Let $M(\phi ,\xi ,\eta ,g_{M})$ be a Kenmotsu manifold\ of
dimension $2m+1$ and $(N,g_{N})$ a Riemannian manifold of dimension $n.$ Let 
$F:M(\phi ,\xi ,\eta ,g_{M})\rightarrow $ $(N,g_{N})$ be an anti-invariant
Riemannian submersion. Then the following assertions are equivalent to each
other:
\end{theorem}

$\ i)$ $(\ker F_{\ast })^{\bot }$ \textit{defines a totally geodesic
foliation on }$M,$

$\ ii)~g_{M}(\mathcal{A}_{X}BY,\phi V)=g_{M}(CY,\phi \mathcal{A}_{X}V),$

$iii)~g_{N}((\nabla F_{\ast })(X,\phi Y),F_{\ast }\phi V)=-g_{M}(CY,\phi 
\mathcal{A}_{X}V),$

\textit{for }$X,Y\in \Gamma ((\ker F_{\ast })^{\bot })$\textit{\ and }$V\in
\Gamma (\ker F_{\ast })$.

\begin{proof}
From (\ref{metric}) and (\ref{nambla}), we obtain%
\begin{equation*}
g_{M}(\nabla _{X}Y,V)=g_{M}(\nabla _{X}\phi Y,\phi V),
\end{equation*}%
for\textit{\ }$X,Y\in \Gamma ((\ker F_{\ast })^{\bot })$\textit{\ }and%
\textit{\ }$V\in \Gamma (\ker F_{\ast }).$By virtue of (\ref{IREM}), we get%
\begin{equation*}
g_{M}(\nabla _{X}Y,V)=g_{M}(\mathcal{\nabla }_{X}BY+\mathcal{\nabla }%
_{X}CY,\phi V).
\end{equation*}%
Using (\ref{3}) and (\ref{IKE4}), we have 
\begin{equation*}
g_{M}(\nabla _{X}Y,V)=g_{M}(\mathcal{A}_{X}BY+\mathcal{V}\nabla _{X}BY,\phi
V)-g_{M}(CY,\phi \mathcal{A}_{X}V).
\end{equation*}%
The last equation shows $(i)\Leftrightarrow (ii)$.

For $X,Y\in \Gamma ((\ker F_{\ast })^{\bot })$\textit{\ }and\textit{\ }$V\in
\Gamma (\ker F_{\ast }),$%
\begin{equation}
g_{M}(\mathcal{A}_{X}BY,\phi V)=g_{M}(CY,\phi \mathcal{A}_{X}V)
\label{bilal}
\end{equation}%
Since differential $F_{\ast }$ preserves the lenghts of horizontal vectors
the relation (\ref{bilal}) forms%
\begin{equation}
g_{M}(CY,\phi \mathcal{A}_{X}V)=g_{N}(F_{\ast }\mathcal{A}_{X}BY,F_{\ast
}\phi V)  \label{K7}
\end{equation}%
By using (\ref{3}) and (\ref{5}) in (\ref{K7}), we obtain%
\begin{equation*}
g_{M}(CY,\phi \mathcal{A}_{X}V)=g_{N}(-(\nabla F_{\ast })(X,\phi Y),F_{\ast
}\phi V)
\end{equation*}%
which tells that $(ii)\Leftrightarrow (iii)$.
\end{proof}

\begin{corollary}
\label{totgeohorcor}Let $F:M(\phi ,\xi ,\eta ,g_{M})\rightarrow $ $(N,g_{N})$
be an anti-invariant Riemannian submersion such that $(\ker F_{\ast })^{\bot
}=\phi \ker F_{\ast }\oplus \{\xi \}$, where $M(\phi ,\xi ,\eta ,g_{M})$ is
a Kenmotsu manifold and $(N,g_{N})$ is a Riemannian manifold. Then the
following assertions are equivalent to each other:

$\ i)$ $(\ker F_{\ast })^{\bot }$ \textit{defines a totally geodesic
foliation on }$M,$

$ii)$ $\mathcal{A}_{X}\phi Y=0,$

$iii)$ $(\nabla F_{\ast })(X,\phi Y)=0$ \textit{for }$X,Y\in \Gamma ((\ker
F_{\ast })^{\bot })$\textit{\ and }$V\in \Gamma (\ker F_{\ast })$.
\end{corollary}

The following result is a consequence from (\ref{1}) and (\ref{IKE2}).

\begin{theorem}
\label{teonottot}Let $F$ be an anti-invariant Riemannian submersion from a
Kenmotsu manifold $M(\phi ,\xi ,\eta ,g_{M})$ to a Riemannian manifold $%
(N,g_{N})$. Then $(\ker F_{\ast })$ does not define a totally geodesic
foliation on $M.$
\end{theorem}

Using Theorem \ref{teonottot}, one can give the following result.

\begin{theorem}
\label{map}Let $F:M(\phi ,\xi ,\eta ,g_{M})\rightarrow $ $(N,g_{N})$ be an
anti-invariant Riemannian submersion where $M(\phi ,\xi ,\eta ,g_{M})$ is a
Kenmotsu manifold and $(N,g_{N})$ is a Riemannian manifold. Then $F$ is not
a totally geodesic map.
\end{theorem}

\begin{remark}
\label{ben}Now we suppose that $\left\{ e_{1},...,e_{m}\right\} $ is a local
orthonormal frame of $\Gamma (\ker F_{\ast })$. From the well known equation 
$H$ $=\frac{1}{m}\sum\limits_{i=1}^{m}\mathcal{T}_{e_{i}}e_{i}$ ,(\ref{1})
and (\ref{4b}) we have%
\begin{eqnarray*}
mg(H,\xi ) &=&g(T_{e_{1}}e_{1},\xi )+g(T_{e_{2}}e_{2},\xi )+\cdots
+g(T_{e_{m}}e_{m},\xi ) \\
&=&-g(T_{e_{1}}\xi ,e_{1})-g(T_{e_{2}}\xi ,e_{2})-\cdots -g(T_{e_{m}}\xi
,e_{m}) \\
&=&-g(e_{1},e_{1})-g(e_{2},e_{2})-\cdots -g(e_{m},e_{m}) \\
&=&-m
\end{eqnarray*}%
We get $g(H,\xi )=-1.$ So $\ker F_{\ast }$ has not minimal fibres.
\end{remark}

By virtue of Remark \ref{ben}, we have the following theorem.

\begin{theorem}
\label{nothar}Let $F:M(\phi ,\xi ,\eta ,g_{M})\rightarrow $ $(N,g_{N})$ be
an anti-invariant Riemannian submersion where $M(\phi ,\xi ,\eta ,g_{M})$ is
a Kenmotsu manifold and $(N,g_{N})$ is a Riemannian manifold. Then $F$ is
not harmonic.
\end{theorem}

\subsection{\textbf{Anti-invariant Riemannian submersions admitting vertical
structure vector field }}

In this subsection, we will prove that there do not exist (anti-invariant)
Riemannian submersions from Kenmotsu manifolds onto Riemannian manifolds
such that characteristic vector field $\xi $ is a vertical vector field.
Moreover, we will give a method to get horizontally conformal submersion
examples from warped product manifolds onto Riemannian manifolds.

It is easy to see that $\mu $ is an invariant distribution of $(\ker F_{\ast
})^{\bot },$ under the endomorphism $\phi $. Thus, for $X\in \Gamma ((\ker
F_{\ast })^{\bot }),$ we have%
\begin{equation}
\phi X=BX+CX,  \label{A2}
\end{equation}%
where $BX\in \Gamma (\ker F_{\ast })$ and $CX\in \Gamma (\mu ).$ On the
other hand, since $F_{\ast }((\ker F_{\ast })^{\bot })=TN$ and $F$ is a
Riemannian submersion, using (\ref{A2}) we derive $g_{N}(F_{\ast }\phi
V,F_{\ast }CX)=0,$ for every $X\in $ $\Gamma ((\ker F_{\ast }))^{\perp \text{
}}$and $V\in \Gamma (\ker F_{\ast })$, which implies that 
\begin{equation}
TN=F_{\ast }(\phi (\ker F_{\ast }))\oplus F_{\ast }(\mu ).  \label{A2a}
\end{equation}

\begin{theorem}
\label{WPC}Let $(M^{m+1}=I\times _{f}L^{m},g_{M}=dt^{2}+f^{2}g_{L})$ be a
warped product manifold of an interval $I$ and a Riemannian manifold $L$. If 
$F:(M^{m+1},g_{M})\rightarrow $ $(N^{n},g_{N})$ is a Riemannian submersion
with vertical vector field $\frac{\partial }{\partial t}=\partial _{t}$ then
warped product manifold is a Riemannian product manifold.
\end{theorem}

\begin{proof}
Let $\sigma =(t,x_{1},x_{2},...,x_{m})$ be a coordinate system for $M$ at $%
p\in M$ and $y_{1},y_{2},...,y_{n}$ be a coordinate system for $N$ at $F(p$%
). Since $\partial _{t}$ is a vertical vector field, we have%
\begin{equation*}
0=F_{\ast }(\partial _{t})_{p}=\sum\limits_{i=1}^{n}\frac{\partial
(y_{i}\circ F)}{\partial _{t}}(p)\frac{\partial }{\partial y_{i}}\mid
_{F(p)}.
\end{equation*}%
So\ the component functions\ $y_{i}\circ F=f_{i}$ of $F$ do not contain $t$
parameter. Namely, 
\begin{equation*}
F:I\times _{f}L\rightarrow N,(t,x)\rightarrow F(t,x)=(f_{1}(x),...,f_{n}(x)),
\end{equation*}%
where $x=(x_{1},x_{2},...,x_{m})$ and also $(\ker F_{\ast })^{\bot }\mid
_{(t,x)}\subseteq T_{(t,x)}(\{t\}\times L)\cong T_{x}L$ at point $p=(t,x)\in
M$. That is, if $\tilde{X}\in $ $(\ker F_{\ast })^{\bot }$, there is a
vector field $X$ $\in \Gamma (TN)$ such that the lift of $X$ to $I\times L$
is the vector field $\tilde{X}$, $\pi _{2\ast }(\tilde{X}_{p})=X_{\pi
_{2}(p)}$ for all $p\in M$. For the sake of \ the simplify we use the same
notation for a vector field and for its lift.

Using Proposition \ref{ONEILL} (ii), we obtain 
\begin{equation}
\nabla _{X}\partial _{t}=\frac{f^{\prime }}{f}X  \label{RIE1}
\end{equation}%
for $X\in \Gamma ((\ker F_{\ast })^{\bot })$. From (\ref{3}) and (\ref{RIE1}%
) we have%
\begin{equation}
\mathcal{A}_{X}\partial _{t}=\frac{f^{\prime }}{f}X  \label{RIE2}
\end{equation}%
for $X\in \Gamma ((\ker F_{\ast })^{\bot })$.

By applying (\ref{TUW2}), (\ref{4c}) and (\ref{RIE2}), we find%
\begin{equation*}
g_{M}(\mathcal{A}_{X}Y,\partial _{t})=-\frac{f^{\prime }}{f}g_{M}(X,Y)=-%
\frac{f^{\prime }}{f}g_{M}(Y,X)=g_{M}(\mathcal{A}_{Y}X,\partial _{t})=-g_{M}(%
\mathcal{A}_{X}Y,\partial _{t})
\end{equation*}%
for \textit{\ }$X,Y\in \Gamma ((\ker F_{\ast })^{\bot })$. Thus, we obtain%
\begin{equation}
g_{M}(\mathcal{A}_{X}Y,\partial _{t})=-\frac{f^{\prime }}{f}g_{M}(X,Y)=0.
\label{RIE3}
\end{equation}%
It follows from (\ref{RIE3}) that  $f^{\prime }=0$ . Hence warping function $%
f$ must be constant. Therefore, up to a change of scale, $M$ is a Riemannian
product manifold.
\end{proof}

\begin{theorem}
\label{WPK}Let $M(\phi ,\xi ,\eta ,g_{M})$ be a Kenmotsu manifold \ of
dimension $2m+1$ and $(N,g_{N})$ is a Riemannian manifold of dimension $n$.
There does not exist a\ Riemannian submersion $F:M(\phi ,\xi ,\eta
,g_{M})\rightarrow $ $(N,g_{N})$ such that characteristic vector field $\xi $
is a vertical vector field.

\begin{proof}
From \cite{kenmotsu} we know that locally a Kenmotsu manifold is a warped
product $I\times _{f}L$ of an interval $I$ and a Kaehler manifold $L$ with
metric $g_{M}=dt^{2}+f^{2}g_{L}$ and warping function $f(t)=se^{t}$, where $s
$ is a positive constant. Let $\xi =\frac{\partial }{\partial t}=\partial
_{t}$ be a vertical vector field. It follows from Theorem \ref{WPC} , $M$ is
a Riemannian product manifold. Since $f(t)=se^{t}$ is not constant , $M$ can
not be a Riemannian product manifold. \ This is a contradiction which
completes the proof of theorem.
\end{proof}
\end{theorem}

\begin{theorem}
\label{Warped} Let $M=M_{1}\times _{f}M_{2}$ be a warped product manifold
with metric $g=g_{1}+f^{2}g_{2}$ , $\pi _{2}:M_{1}\times M_{2}\rightarrow
M_{2}$ second canonical projection and $(M_{3},g_{3})$ Riemannian manifold.
If $f_{1}$ is a Riemannian submerison from $M_{2}$ onto $M_{3}$ then $%
f_{2}=f_{1}\circ \pi _{2}:$ $M\rightarrow M_{3}$ is a horizontally conformal
submersion.
\end{theorem}

\begin{proof}
Since $f_{1}$ is a Riemannian submersion, rank $f_{1}=\dim M_{3}$. Using
Proposition \ref{Clark}, we have \ rank $f_{2\mid _{(p,q)}}=$ rank $%
f_{1}\mid _{f_{1(q)}}=\dim M_{3}$ for any point $(p,q)$ $\in $ $M.$
Consequently $f_{2}$ is a submersion. Since $\pi _{2}$ is a natural
horizontally conformal submersion for a warped product manifold, we get $%
\ker \pi _{2\ast \mid _{(p,q)}}=T_{(p,q)}M_{1}\equiv T_{(p,q)}(M_{1}\times
\{q\})\cong T_{p}M_{1}.$ So $\ker f_{2\ast \mid _{(p,q)}}\cong
T_{p}M_{1}\times \ker f_{1\ast q}$ and ($\ker f_{2\ast })_{\mid
_{(p,q)}}^{\perp }=\{p\}\times (\ker f_{1\ast })_{\mid q}^{\perp }\cong
(\ker f_{1\ast })_{\mid q}^{\perp }.$ Hence,%
\begin{eqnarray*}
g(X,Y) &=&f^{2}(p)g_{2}(\pi _{2\ast }(X),\pi _{2\ast }(Y)) \\
&=&f^{2}(p)g_{3}(f_{1\ast }(\pi _{2\ast }(X)),f_{1\ast }(\pi _{2\ast }(Y)) \\
&=&f^{2}(p)g_{3}(f_{2\ast }(X),f_{2\ast }(Y))
\end{eqnarray*}%
for $X,Y\in \Gamma ((\ker f_{2\ast })^{\bot }).$ So we get the requested
result.
\end{proof}

\begin{remark}
Theorem \ref{Warped} gives a chance to produce horizontally conformal
submersion examples.
\end{remark}

\section{Examples}

We now give some examples for anti-invariant submersion and anti-invariant
horizontally conformal submersions from Kenmotsu manifolds.\ \ \ 

\begin{example}
\label{ex2}Let $M$ be a Kenmotsu manifold as in Example \ref{KEN}$.$ Let $N$
be $%
\mathbb{R}
\times _{e^{z}}%
\mathbb{R}
^{2}$ The Riemannian metric tensor field $g_{N\text{ }}$is defined by$%
~g_{N}=e^{2z}(du\otimes du+dv\otimes dv)+dz\otimes dz$ on $N$.

Let $F:M\rightarrow N$ be a map defined by $F(x_{1},x_{2},y_{1},y_{2},t)=(%
\frac{x_{1}+y_{2}}{\sqrt{2}},\frac{x_{2}+y_{1}}{\sqrt{2}},z)$. Then, a
simple calculation gives%
\begin{equation*}
\ker F_{\ast }=span\{V_{1}=\frac{1}{\sqrt{2}}(E_{2}-E_{3}),\text{ }V_{2}=%
\frac{1}{\sqrt{2}}(E_{1}-E_{4})\}
\end{equation*}%
and%
\begin{equation*}
(\ker F_{\ast })^{\bot }=span\{H_{1}=\frac{1}{\sqrt{2}}(E_{1}+E_{4}),\text{ }%
H_{2}=\frac{1}{\sqrt{2}}(E_{2}+E_{3}),\text{ }H_{3}=E_{5}=\xi \}.
\end{equation*}%
Then it is easy to see that $F$ is a Riemannian submersion. Moreover, $\phi
V_{1}=-H_{1}$, $\phi V_{2}=-H_{2}$ imply that $\phi (\ker F_{\ast })\subset
(\ker F_{\ast })^{\bot }=$ $\phi (\ker F_{\ast })\oplus \{\xi \}$. Thus $F$
is an anti-invariant Riemannian submersion such that $\xi $ is a horizontal
vector field.
\end{example}

\begin{example}
\label{ex1}Let $M$ \ be a Kenmotsu manifold as in Example \ref{KEN} and $N$
be $%
\mathbb{R}
^{2}.$ The Riemannian metric tensor field $g_{N\text{ }}$is defined by $%
g_{N}=e^{2z}(du\otimes du+dv\otimes dv)$ on $N$.

Let $F:M\rightarrow N$ be a map defined by $F(x_{1},x_{2},y_{1},y_{2},z)=(%
\frac{x_{1}+y_{2}}{\sqrt{2}},\frac{x_{2}+y_{1}}{\sqrt{2}})$. Then, by direct
calculations we have 
\begin{equation*}
\ker F_{\ast }=span\{V_{1}=\frac{1}{\sqrt{2}}(E_{3}-E_{2}),\text{ }V_{2}=%
\frac{1}{\sqrt{2}}(E_{4}-E_{1}),\text{ }V_{3}=E_{5}=\xi =\frac{\partial }{%
\partial z}\}
\end{equation*}%
and 
\begin{equation*}
(\ker F_{\ast })^{\bot }=span\{H_{1}=\frac{1}{\sqrt{2}}(E_{3}+E_{2}),\text{ }%
H_{2}=\frac{1}{\sqrt{2}}(E_{4}+E_{1})\}.
\end{equation*}%
Then it is easy to see that $F$ is a horizontally conformal submersion.
Moreover, $\phi V_{1}=H_{2},$ $\phi V_{2}=H_{1},$ $\phi V_{3}=0$ imply that $%
\phi (\ker F_{\ast })=(\ker F_{\ast })^{\bot }.$ As a result, $F$ is an
anti-invariant horizontally conformal submersion such that $\xi $ is a
vertical vector field.
\end{example}

\begin{remark}
Recently Akyol M.A. and \c{S}ahin B. \cite{AS} studied conformal
anti-invariant submersions from almost Hermitian manifolds onto Riemannian
manifolds. So it will be worth the study area which is anti-invariant
(horizontally) conformal submersion from almost contact metric manifolds
onto Riemannian manifolds.
\end{remark}

\textbf{Acknowledgement}

The authors are grateful to the referee for his/her valuable comments and
suggestions.


\begin{thebibliography}{99}
\bibitem{AS} {\small Akyol M. A, \c{S}ahin B.,} {\small \emph{Conformal
anti-invariant submersions from almost Hermitian manifolds, }Turk J. Math.,
DOI: 10.3906/mat-1408-20.}

\bibitem{B} {\small Baird P., Wood J.C.}, {\small \emph{Harmonic Morphisms
Between Riemannian Manifolds,}} {\small London Mathematical Society
Monographs, 29, Oxford University Press, The Clarendon Press, Oxford, 2003.}

\bibitem{BLAIR} {\small Blair D. E., \emph{Riemannian geometry of contact
and symplectic manifolds,} Progress in Mathematics. 203, Birkhauser Boston,
Basel, Berlin, 2002.}

\bibitem{BL} {\small Bourguignon J. P., Lawson H. B.,} {\small \emph{%
Stability and isolation phenomena for Yang-mills fields, }Commun. Math.
Phys. 79, (1981), 189-230.}

\bibitem{BOL} {\small Bourguignon J. P., Lawson H. B.,} {\small \emph{A
Mathematician's visit to Kaluza-Klein theory,}} {\small Rend. Semin. Mat.
Torino Fasc. Spec. (1989), 143-163.}

\bibitem{BRICLC} {\small Brickell F., Clark R. S., \emph{Differentiable
Manifolds, } London ; New York : V. N. Reinhold Co., (1970).}

\bibitem{CHEN} {\small Chen B. Y.,} {\small \emph{Geometry of slant
submanifolds,} Katholieke Universiteit Leuven, Leuven, 1990.}

\bibitem{FAL} {\small Falcitelli M., Ianus S., and Pastore A. M.}$,{\small 
\emph{Riemannian}}$ {\small \emph{\ submersions and related topics,}}, 
{\small World Scientific Publishing Co., 2004.}

\bibitem{GRAY} {\small Gray A.,} {\small \emph{Pseudo-Riemannian almost
product manifolds and submersions,} J. Math. Mech, 16 (1967), 715-737.}

\bibitem{IV} {\small Ianus S., Visinescu M.,} {\small \emph{\ Kaluza-Klein
theory with scalar fields and generalized Hopf manifolds,}} {\small Class.
Quantum Gravity 4, (1987), 1317-1325.}

\bibitem{IV2} {\small Ianus S., Visinescu M., \emph{\ Space-time
compactification and Riemannian submersions,}} {\small In: Rassias, G.(ed.)
The Mathematical Heritage of C. F.} {\small Gauss, (1991), 358-371, World
Scientific, River Edge.}

\bibitem{IANUS2} {\small Ianus S., Mazzocco R., Vilcu G. E.,} {\small \emph{%
Riemannian submersions from quaternionic manifolds}}, {\small Acta Appl.
Math. 104, (2008) 83-89.}

\bibitem{IANUS} {\small Ianus S., Ionescu A. M., Mazzocco R., Vilcu G. E.,} 
{\small \emph{Riemannian submersions from almost contact metric manifolds,}} 
{\small Abh. Math. Semin. Univ. Hambg. 81, (2011), 101--114}

\bibitem{kenmotsu} {\small Kenmotsu, K.,} {\small \emph{A class of almost
contact Riemannian manifolds,}} {\small Tohoku Math. J., 24(2), (1972),
93-103.}

\bibitem{cm} {\small Murathan C., Kupeli Erken I.,} {\small \emph{%
Anti-invariant Riemannian submersions from cosymplectic manifolds,}} {\small %
Filomat, 29(7), (2015).}

\bibitem{BO1} {\small O'Neill B., \emph{The fundamental equations of
submersion,} Michigan Math. J. 13, (1966) 459-469.}

\bibitem{BO2} {\small O'Neill B., \emph{Semi-Riemannian geometry with
applications to relativity,} Academic Press, New York-London 1983.}

\bibitem{ORN} {\small Ornea L., Romani G., \emph{The fundamental equations
of conformal submersions, } Beitr\"{a}ge Algebra Geom. 34 no. 2 (1993),
233--243.}

\bibitem{PARK1} {\small Park K. S.,} {\small \emph{H-slant submersions}},%
{\small \ Bull. Korean Math. Soc. 49(2), (2012), 329-338.}

\bibitem{PARK2} {\small Park K. S.,} {\small \emph{H-semi-invariant
submersions,}} {\small Taiwanese J. Math., 16(5), (2012), 1865-1878.}

\bibitem{PON} {\small Ponge R., Reckziegel H., \emph{Twisted products in
pseudo--Riemannian geometry, }Geom. Dedicate, 48(1), (1993), 15-25.}

\bibitem{SAHIN} {\small \c{S}ahin B.,} {\small \emph{Anti-invariant
Riemannian submersions from almost Hermitian manifolds}}, {\small Cent. Eur.
J. Math. 8(3), (2010) 437-447.}

\bibitem{SAHIN1} {\small \c{S}ahin B.,} {\small \emph{Slant submersions from
almost Hermitian manifolds}}, {\small Bull. Math. Soc. Sci. Math. Roumanie
Tome 54(102) No. 1, (2011), 93 - 105.}

\bibitem{SAHIN3} S{\small ahin B., \emph{Riemannian submersions from almost
Hermitian manifolds}}, {\small Taiwanese J. Math., 17, No. 2, (2012),
629-659.}

\bibitem{SAHIN2} {\small \c{S}ahin B., \emph{Semi-invariant submersions from
almost Hermitian manifolds}}, {\small Canad. Math. Bull., 56(1), (2013),
173-183.}

\bibitem{WAT} {\small Watson B.,} {\small \emph{Almost Hermitian submersions,%
}} {\small J. Differential Geom., 11(1), (1976), 147-165.}

\bibitem{WATSON} {\small Watson B., \emph{G, G}}$^{^{\prime }}$-{\small 
\emph{Riemannian submersions and nonlinear gauge field equations of general
relativity,}} {\small In: Rassias, T. (ed.) Global Analysis -} {\small %
Analysis on manifolds, dedicated M. Morse. Teubner-Texte Math., 57 (1983),
324-349, Teubner, Leipzig.}

\bibitem{MUS} {\small M. T. Mustafa, \emph{\ Applications of harmonic
morphisms to gravity,} J. Math. Phys., 41(10), (2000), 6918-6929.}
\end{thebibliography}
\end{document}